\documentclass[10pt]{amsart}
\usepackage[cp1251]{inputenc}
\usepackage[english]{babel}

\usepackage{amsmath,amsthm,amsfonts,amssymb}
\usepackage{cite}

\newtheorem{lemma}{Lemma}
\newtheorem{theorem}{Theorem}
\newtheorem{definition}{Definition}

\newtheorem{proposition}{Proposition}
\newtheorem{remark}{Remark}
\newcommand{\alert}[1]{{\bf #1}}
\newcommand{\Lang}{\mathcal{L}}
\newcommand{\A}{\mathcal{A}}

\newcommand{\Auniverse}{A}

\newcommand{\ox}{\overline{x}}

\title{Equationally Noetherian property of Ershov algebras}
\author{Y.~Dvorzhetskiy}

\begin{document}

\begin{abstract}
This article is about equationally Noetherian and weak equationally Noetherian property of Ershov algebras. Here we show two canonical forms of the system of equations over Ershov algebras and two criteria of equationally Noetherian and weak equationally Noetherian properties.

\end{abstract}

\maketitle


\section*{Introduction}

E.~Daniyarova, A.~Myasnikov and V.~Remeslennikov proved two Unification Theorems in the articles \cite{DMR1,DMR2,DMR3,DMR4}. These theorems give us 7 equivalent approaches to describe all coordinate algebras of algebraic structures. These theorems can be applied to the concrete classes of structures only. Such classes are $\mathbf{N}$ and $\mathbf{N'}$~--- classes of equationally Noetherian and weak equationally Noetherian structures, $\mathbf{Q}$ and $\mathbf{U}$~--- classes of $q_\omega$- and $u_\omega$-compact structures. It is need to prove that structure is in one of these classes to apply Unification Theorems.

A.~Shevlyakov proved criteria of equationally Noetherian, weak equationally Noetherian property, $q_\omega$- and $u_\omega$-compactness of Boolean algebras with constants \cite{Shevlyakov}. The criteria of equationally Noetherian property are also true over more general distributive lattices \cite{Dv_dist_lattices}.

Here we study Ershov algebras~--- they are something mean between distributive lattices and Boolean algebras. We will show two canonical forms for systems of equations and prove two criteria of equationally an weak equationally Noetherian property.


\section{Lattices, Ershov algebras}

We show in this section all required definitions from lattice theory. Also we define Ershov algebras. All required definitions from algebraic geometry can be found in \cite{DMR1,DMR2,DMR3}.

Let $\Lang_0 = \{ \vee^{(2)}, \wedge^{(2)} \}$~--- the first-order language with two binary functional symbol: $\vee$ and $\wedge$.

\begin{definition}[Lattice]
Algebraic system $\A = \left< \Auniverse; \vee, \wedge \right>$ of first-order language $\Lang_0$ is called \alert{lattice}, if for all $a, b, c \in \Auniverse$:
\begin{enumerate}
\item $a \wedge a = a$, $a \vee a = a$.
\item $a \wedge b = b \wedge a$, $a \vee b = b \vee a$.
\item $(a \wedge b) \wedge c = a \wedge (b \wedge c)$, $(a \vee b) \vee c = a \vee (b \vee c)$.
\item $a \wedge (a \vee b) = a$, $a \vee (a \wedge b) = a$.
\end{enumerate}  
\end{definition}

Next we give definition of distributive lattices

\begin{definition}[Distributive lattice]
Lattice $\A = \left< \Auniverse; \vee, \wedge \right>$ is called  \alert{distributive}, if for all $a, b, c \in \Auniverse$:
$$
a \wedge (b \vee c) = (a \wedge b) \vee (a \wedge c),
$$
$$
a \vee (b \wedge c) = (a \vee b) \wedge (a \vee c).
$$
\end{definition}

A greatest element in lattice, if it exists, denotes as 1, and a least element denotes as 0, respectively. Lattices with 0 and 1 are called \alert{bounded lattices}.

\begin{definition}[Complement]
Two elements $x$ and $x'$ of bounded lattice are \alert{comp\-le\-ments} if $x \wedge x' = 0$ and $x \vee x' = 1$.
\end{definition}

A complement $x'$ of element $x$ denotes as $\overline{x}$.

\begin{definition}[Complement in interval]
We called that element $x'$ is \alert{complement of $x$ in interval $[ a, b ]$} if $x' \wedge x = a$ and $x' \vee x = b$.
\end{definition}

We note that for all distributive littices there is at most one complement in interval \cite{Birkhoff}.

It is possible to define another complement in a lasttice with least element 0.

\begin{definition}[Relative complement]
Relative complement of $a$ in $b$ is a complement of $a$ in interval $[ 0, a \vee b ]$.
\end{definition}

Such complemlent of $a$ in $b$ we denote as $b \setminus a$ and call \alert{difference} between $b$ and $a$.

\begin{definition}[Ershov algebras {\cite{Ershov}}]
Distributive lattice 
$$\A = \left< \Auniverse, \vee^{(2)}, \wedge^{(2)}, \setminus^{(2)}, 0 \right>$$
with least element 0 and operation of relative complement we call \alert{Ersov algebra}.
\end{definition}

If there is a graeatest element 1 in any Ershov algebra then this Ershov algebra is simply Boolean algebra:
$$
\ox = 1 \setminus x.
$$

Let's show some propositions for Ershov algebras
\begin{proposition}
Let $x$ and $y$ are arbitrary elements of Ershov algebra. The following equivalences are true:
\begin{enumerate}
\item $x \setminus y \leq x$,
\item $(x \setminus y) \vee x = x$,
\item $(x \setminus y) \vee (x \wedge y) = x$, 
\item $(x \setminus y) \vee y = x \vee y$,
\item $(x \setminus y) \wedge x = x \setminus y$,
\item $(x \setminus y) \wedge y = 0$.
\end{enumerate}
\end{proposition}
\begin{proof}
These propositions follow clearly from the given definiton of a relative complement.
\end{proof}

Next, prove another equivalences.
\begin{proposition}\label{main_prop}
The following equivalences are true in any Ershov algebra:
\begin{enumerate}
\item $(x \vee y) \setminus a = (x \setminus a) \vee (y \setminus a)$,
\item $(x \wedge y) \setminus a = (x \setminus a) \wedge (y \setminus a)$,
\item $(x \setminus y) \setminus a = (x \setminus y) \vee (x \setminus a)$.
\end{enumerate}
and the following ones are true as well:
\begin{enumerate}
\item $x \setminus (a \vee b) = (x \setminus a) \wedge (x \setminus b)$,
\item $x \setminus (a \wedge b) = (x \setminus a) \vee (x \setminus b)$,
\item $x \setminus (a \setminus b) = (x \setminus a) \vee (x \wedge a \wedge b)$.
\end{enumerate}
\end{proposition}
\begin{proof}
Let's prove that equations by the definition of a relative complement.
\begin{enumerate}
\item
\begin{equation*}
\begin{array}{c}
\left((x \setminus a) \vee (y \setminus a) \right) \vee a = \left((x \setminus a) \vee a\right) \vee \left((y \setminus a) \vee a \right) = \\

=  (x \vee a) \vee (y \vee a) = (x \vee y) \vee a,
\end{array}
\end{equation*}
and
$$\left((x \setminus a) \vee (y \setminus a) \right) \wedge a = \left((x \setminus a) \wedge a\right) \vee \left((y \setminus a) \wedge a \right) = 0 \vee 0 = 0.$$
It implies,
$$
\left((x \setminus a) \vee (y \setminus a) \right) \vee a = (x \vee y) \vee a,
$$
$$\left((x \setminus a) \vee (y \setminus a) \right) \wedge a = 0,
$$
and $(x \setminus a) \vee (y \setminus a)$ is the complement of $a$ in interval $[0, (x \vee y) \vee a]$, i.e. $(x \vee y) \setminus a = (x \setminus a) \vee (y \setminus a)$.

\item This equivalence is proved similarly.

\item 
\begin{equation*}
\begin{array}{c}
((x \setminus y) \wedge (x \setminus a)) \vee a = ((x \setminus y) \vee a) \wedge ((x \setminus a) \vee a) = \\

((x \setminus y) \vee a) \wedge (x \vee a) = (x \setminus y \wedge x) \vee a = (x \setminus y) \wedge a. \\

((x \setminus y) \wedge (x \setminus a)) \wedge a = (x \setminus y) \wedge 0 = 0.
\end{array}
\end{equation*}

Then, we proved that $(x \setminus y) \setminus a = x \setminus (y \vee a)$.
\end{enumerate}

Let's prove other equivalences.
\begin{enumerate}
\item Clearly:
\begin{equation*}
\begin{array}{c}
((x \setminus a) \wedge (x \setminus b)) \vee (a \vee b) = ((x \setminus a) \vee a \vee b) \wedge ((x \setminus b) \vee a \vee b) = \\

= (x \vee a \vee b) \wedge (x \vee a \vee b) = x \vee (a \vee b).
\end{array}
\end{equation*}
\begin{equation*}
\begin{array}{c}
((x \setminus a) \wedge (x \setminus b)) \wedge (a \vee b) = ((x \setminus a) \wedge a \wedge (x \setminus b)) \vee ((x \setminus a) \wedge (x \setminus b) \wedge b) = \\

= (0 \wedge (x \setminus b)) \vee ((x \setminus a) \wedge 0) = 0 \vee 0 = 0.
\end{array}
\end{equation*}
\item This equivalence is proved similarly.
\item By definiton.
\begin{equation*}
\begin{array}{c}
(x \setminus a) \vee (x \wedge a \wedge b) \vee  (a \setminus b) =
(x \setminus a) \vee (x \vee (a \setminus b)) \wedge (a \wedge b \vee (a \setminus b)) = \\

= (x \setminus a) \vee (x \vee (a \setminus b)) \wedge a =
((x \setminus a) \vee x \vee (a \setminus b)) \wedge (x \vee (a \setminus b) \vee a) = \\

= (x \vee (a \setminus b)) \wedge (x \vee a \vee b) = 
x \vee ((a \setminus b) \wedge (a \vee b)) = \\

= x \vee ((a \setminus b) \wedge a) \vee ((a \setminus b) \wedge b) =
x \vee ((a \setminus b) \vee 0) = x \vee (a \setminus b).
\end{array}
\end{equation*}
And:
$$((x \setminus a) \vee (x \wedge a \wedge b)) \wedge (a \setminus b) = 
((x \setminus a) \wedge (a \setminus b)) \vee (x \wedge a \wedge b \wedge (a\setminus b)) = 0 \vee 0 = 0.$$
\end{enumerate}
\end{proof}


\section{Canonical form of systems of equations}

We construct here the canonical form of system of equations in Ershov algebras.

\begin{lemma}\label{lemma_term_nf}
Any term $t (\overline{x})$ of variables $\overline{x} = (x_1, \ldots, x_n)$ in Ersov algebra $\A$ can be rewritten as:
$$
\left( A_{1,1} \wedge A_{1,2} \wedge \ldots \wedge A_{1,k_1} \right) \vee \ldots \vee \left( A_{m,1} \wedge \ldots \wedge A_{m, k_m} \right),
$$
where $A_{i,j}$ is either a variable symbol or a constant symbol of the language $\Lang$, or expresion $a \setminus b$, where $a$ and $b$ are also either a variable or a constant symbol. And whole expression is written in DNF.
\end{lemma}
\begin{proof}
Let $t(\overline{x})$ is an arbitrary term of variables $\overline{x}$. If there is no symbol $\setminus$ in this term, then rewrite this term in DNF, where $A_{i,j}$ either a variable symbol or a constant symbol.

If there are symbols $\setminus$, but the left and right operands haven't this one, then we also have required from.

Suppose that the term with non-trivial subterm exists. We can write it as:
$$(s_1 \vee s_2) \setminus t$$
or
$$(s_1 \wedge s_2) \setminus t$$
or
$$(s_1 \setminus s_2) \setminus t,$$
where $s_1$, $s_2$, $t$~--- terms of variables $\overline{x}$. If we apply Proposition \ref{main_prop} we get:
$$s_1 \setminus t \vee s_2 \setminus t$$
or
$$s_1 \setminus t \wedge s_2 \setminus t$$
or
$$s_1 \setminus s_2 \vee s_1 \setminus t.$$
Denote, that after appling we have less symbols $\setminus$ in the left operand of the equivalented subterm. Appling these equivalences many times we can rewrite term in form where is no $\setminus$ symbols in the left operand.

Analogically, if start term has one of the following forms:
$$t \setminus (s_1 \vee s_2)$$
or
$$t \setminus (s_1 \wedge s_2)$$
or
$$t \setminus (s_1 \setminus s_2),$$
where $s_1$, $s_2$, $t$~--- terms of variables $\overline{x}$. If we apply Proposition \ref{main_prop} we get:
$$(t \setminus s_1) \wedge (t \setminus s_2)$$
or
$$(t \setminus s_1) \vee (t \setminus s_2)$$
or
$$(t \setminus s_1) \vee (t \wedge s_1 \wedge s_2),$$
respectively. Denote, that after appling we have less symbols $\setminus$ in the right operand of the equivalented subterm, and left operand does not change.

Do this actions many times we can rewrite the term to the form which is required.
\end{proof}

\begin{remark}\label{remark_to_lemma_term_nf}
Denote that we can rewrite term in CNF as well as DNF. It will be applied below.

Also denote that if the subterm $A_{i,j}$ does not contain the symbol $\setminus$, then we can also rewrite this subterm as $A_{i,j} = A_{i,j} \setminus 0$. Therefore, we can assume that all $A_{i,j}$ written as $a \setminus b$, where $a$ and $b$ either a variable or functional symbol.
\end{remark}

It implies next theorem.

\begin{theorem}\label{theorem_nf}
Any system of equations $S(\ox)$ of variables $\ox$ is equivalent to the system $S'(\ox)$, consisted from following equations
\begin{equation*}
\begin{array}{c}
\left( A_{1,1} \wedge \ldots \wedge A_{1,k_1} \right) \vee \ldots \vee \left( A_{m,1} \wedge \ldots \wedge A_{m, k_m} \right) = \\
= \left( B_{1,1} \wedge \ldots \wedge B_{1,k_1} \right) \vee \ldots \vee \left( B_{1,1} \wedge \ldots \wedge A_{s, l_s} \right),
\end{array}
\end{equation*}
where $A_{i,j}$ and $B_{i,j}$ are either a variable or constant symbol of the language $\Lang$, or a subterm $a \setminus b$, where $a$ and $b$ are either a variable or constant symbol, and the whole ewxpression is written in DNF.
\end{theorem}


\section{Equationally noetherian property}

All required in this setion definitons can be found in \cite{DMR1, DMR2, DMR3, DMR4}.

\begin{theorem}
$\mathcal{C}$-algebra $\A$ is equationally noetherian if and only if subalgebra $\mathcal{C}$ is finite. 
\end{theorem}
\begin{proof}
Suppose that the subalgebra $\mathcal{C}$ is finite. Theorem \ref{theorem_nf} implies that the number of non-equivalent equations and non-equivalent systems of equations are finite. Therefore we can clear the infinite system of equitions from all equivalent systems and make some finite system of equations. Then Ershov algebra $\A$ is equationally Noetherian.

Prove to the other side. Let $\mathcal{C}$ is infinite. Because an Ershov algebra is a distributive lattice, then we can apply the proof from  \cite{Dv_dist_lattices}. Finally, if $\mathcal{C}$ is finite, then $\A$ does not have equationally Noetherian property.
\end{proof}

Also prove following proposition by ananlogy with Proposition \ref{main_prop}.

\begin{proposition}\label{prop_n_addons}
In arbitratry Ershov algebra the following equivalences are true:
\begin{enumerate}
\item $a_1 \setminus b_1 \wedge \ldots \wedge a_n \setminus b_n =
(a_1 \wedge \ldots \wedge a_n) \setminus (b_1 \vee \ldots \vee b_n),$
\item $a_1 \setminus b_1 \vee \ldots \vee a_n \setminus b_n =
(a_1 \vee \ldots \vee a_n) \setminus (b_1 \wedge \ldots \wedge b_n).$
\end{enumerate}
\end{proposition}
\begin{proof}
Here will be proved only equivalences for $n=2$. Any equivalence for $n > 0$ clearly true by the induction.

Prove first proposition by the defintion of relative complement:
\begin{equation*}
\begin{array}{c}
(a_1 \setminus b_1) \wedge (a_2 \setminus b_2) \vee (b_1 \vee b_2) = ((a_1 \setminus b_1) \vee b_1 \vee b_2) \wedge ((a_2 \setminus b_2) \vee b_1 \vee b_2) = \\

(a_1 \vee b_1 \vee b_2) \wedge (a_2 \vee b_1 \vee b_2) = (a_1 \wedge a_2) \vee (b_1 \vee b_2),
\end{array}
\end{equation*}
and
\begin{equation*}
\begin{array}{c}
(a_1 \setminus b_1) \wedge (a_2 \setminus b_2) \wedge (b_1 \vee b_2) = \\

= ((a_1 \setminus b_1) \wedge b_1 \wedge (a_2 \setminus b_2)) \vee ((a_1 \setminus a_2) \wedge (a_2 \setminus b_2) \wedge b_2) = \\

= (0 \wedge (a_2 \setminus b_2)) \vee ((a_1 \setminus b_1) \wedge 0) = 0 \vee 0 = 0.
\end{array}
\end{equation*}
Finally, $(a_1 \setminus b_1) \wedge (a_2 \setminus b_2)$ is a complement of $b_1 \vee b_2$ in interval $[0, (a_1 \wedge a_2) \vee (b_1 \vee b_2)]$, i.e.:
$$
a_1 \setminus b_1 \wedge a_2 \setminus b_2 = (a_1 \wedge a_2) \setminus (b_1 \vee b_2).
$$

Second proposition is proved similarly.
\end{proof}

In the article \cite{Dv_dist_lattices} wil be proved following proposition.

\begin{proposition}\label{prop_dist_nf}
In any distributive $\mathcal{C}$-lattice any system of equation $S(\ox)$ of varibles $\ox$ can be rewritten to the equivalent system of equation in the following form:
$$
x_{i_1} \wedge \ldots \wedge x_{i_m} \wedge c_a \leq  x_{j_1} \vee \ldots \vee x_{j_l} \vee c_b,
$$
where all variable symbols are different in both sides of the equation. And there is most one constant symbol in each side of the equation, and if there are, then $c_a > c_b$.
\end{proposition}

\begin{remark}
Denote, that $X \leq Y$ we also call an equtions. Really, inequalities in lattices we can also assume as equations:
$$
X \leq Y \quad \sim \quad X \vee Y = Y.
$$
\end{remark}

Let's show another portion of equivalences.

\begin{proposition}\label{prop_equiv_systems}
In any Ershov algebras we have:
\begin{enumerate}
\item 
\begin{equation*}
a \setminus b \leq c \sim a \leq b \vee c,
\end{equation*}

\item 
\begin{equation*}
a \leq b \setminus c \sim
\left\{
\begin{array}{l}
a \leq b, \\
a \wedge c \leq 0.
\end{array}
\right.
\end{equation*}

\end{enumerate}
\end{proposition}

We have all to show another canonical form of the systems of equations over Ershov algebras.

\begin{theorem}\label{theorem_system_nf}
Any system of equations $S(\ox)$ of variables $\ox$ in Ershov $\mathcal{C}$-algebra can be rewritten to the system consisted from the following equations:
$$
x_{i_1} \wedge \ldots \wedge x_{i_m} \wedge c \leq x_{j_1} \vee \ldots \vee x_{j_l}
$$
or (if there is no variables in the right side)
$$
x_{i_1} \wedge \ldots \wedge x_{i_m} \wedge c = 0
$$
or
$$
x_{i_1} \wedge \ldots \wedge x_{i_m} \leq x_{j_1} \vee \ldots \vee x_{j_l}
\vee c
$$
or
$$
x_{i_1} \wedge \ldots \wedge x_{i_m} \leq x_{j_1} \vee \ldots \vee x_{j_l},
$$
where all variable symbols are different in both side of the equation and $c > 0$.
\end{theorem}
\begin{proof}
Let $t(\overline{x}) = s(\ox)$ is the arbitrary equation from $S(\ox)$. Rewrite it as system of two equations:
\begin{equation*}
t(\ox) = s(\ox) \sim
\left\{
\begin{array}{l}
t(\ox) \leq s(\ox),\\
s(\ox) \leq t(\ox).
\end{array}
\right.
\end{equation*}

Do it for each equation in $S(\ox)$ and we will have the system of equations $t(\ox) \leq s(\ox)$, where $t(\ox)$ and $s(\ox)$ are terms of the language $\Lang$ and variables $\overline{x}$. Let $t(\ox) \leq s(\ox)$ is anbitratry equation from the system, rewrite it by the Lemma \ref{lemma_term_nf}:
$$
\left( A_{1,1} \wedge A_{1,2} \wedge \ldots \wedge A_{1,k_1} \right) \vee \ldots \vee \left( A_{m,1} \wedge \ldots \wedge A_{m, k_m} \right) \leq s(\ox),
$$
where each $A_{i,j}$ from Remark \ref{remark_to_lemma_term_nf} is $a \setminus b$, where $a$ and $b$ are either a variable or a constant symbol

Clearly: 
\begin{equation*}
\begin{array}{c}
\left( A_{1,1} \wedge A_{1,2} \wedge \ldots \wedge A_{1,k_1} \right) \vee \ldots \vee \left( A_{m,1} \wedge \ldots \wedge A_{m, k_m} \right) \leq s(\ox) \sim \\

\sim
\left\{
\begin{array}{l}
A_{1,1} \wedge A_{1,2} \wedge \ldots \wedge A_{1,k_1} \leq s(\ox), \\
\ldots \\
A_{m,1} \wedge A_{m,2} \wedge \ldots \wedge A_{m,k_m} \leq s(\ox). \\
\end{array}
\right.
\end{array}
\end{equation*}
Therefore, we can see at only following equations:
$$A_1 \wedge A_2 \wedge \ldots \wedge A_k \leq s(\ox),$$
and more precisely by Remark \ref{remark_to_lemma_term_nf}:
$$
a_1 \setminus b_1 \wedge a_2 \setminus b_1 \wedge \ldots \wedge a_k \setminus b_k \leq s(\ox).
$$
Rewrite it by Proposition \ref{prop_n_addons}:
$$
(a_1 \wedge a_2 \wedge \ldots \wedge a_k) \setminus (b_1 \vee b_2 \vee \ldots \vee b_k) \leq s(\ox).
$$ 
And by Propoposition \ref{prop_equiv_systems}:
$$
(a_1 \wedge a_2 \wedge \ldots \wedge a_k) \leq s(\ox) \vee b_1 \vee b_2 \vee \ldots \vee b_k,
$$
where we have no symbols $\setminus$ in the left side of the equation. And the number of symbols $\setminus$ in the right side is the same.
Do these steps for every equation in the system and have system without symbols $\setminus$ in the left sides.

Do these steps for rigth sides: make CNF, then apply other equivalences in Proposition \ref{prop_n_addons} and \ref{prop_equiv_systems}, and make equiations without symbols $\setminus$.

By Proposition \ref{prop_dist_nf} make all equations as:
$$
x_{i_1} \wedge \ldots \wedge x_{i_m} \wedge c_a \leq x_{j_1} \vee \ldots \vee x_{j_l} \vee c_b.
$$
If we have at most one constant symbol in an arbitrary equation, then we have required normal form, otherwise rewrite equation:
\begin{equation*}
\begin{array}{c}
x_{i_1} \wedge \ldots \wedge x_{i_m} \wedge c_a \leq  x_{j_1} \vee \ldots \vee x_{j_l} \vee c_b \sim \\

\sim \left( x_{i_1} \wedge \ldots \wedge x_{i_m} \wedge c_a \right) \setminus c_b \leq  x_{j_1} \vee \ldots \vee x_{j_l} \sim \\

\sim x_{i_1} \wedge \ldots \wedge x_{i_m} \wedge (c_a \setminus c_b) \leq  x_{j_1} \vee \ldots \vee x_{j_l}.
\end{array}
\end{equation*}
Denote $c = c_a \setminus c_b$, we make required normal form.

Denote that form
$$
x_{i_1} \wedge \ldots \wedge x_{i_m} \wedge c = 0
$$
can be maked if the right side does not have any variable symbol.

Finally, whole system has been rewritten in required form.
\end{proof}


\section{Weak equationally Noetherian property}

Prove for the first some required propositions.

\begin{proposition}
Let $\A$ is an Ershov $\mathcal{C}$-algebra. If $\A$ is weak equationally noetherian, then for all upper-unbound set of elements $\{ c_j | j \in J\} \subseteq \mathcal{C}$ and for some $c$ following equivalence is true:
$$
\left\{ x \wedge c_j = 0 | j \in J \right\} \, \sim \, x \leq c.
$$
\end{proposition}
\begin{proof}
Denote, that if Ershov subalgebra $\mathcal{C}$ is upper-bounded, that there is no unbounded sets. Let $\{ c_j | j \in J\} \subseteq \mathcal{C}$ is upper unbounded set of constants. Consider following system of equations:
$$
S(x) = \left\{ x \wedge c_j = 0 \right\}.
$$
Because $\A$ is weak equationally Noetherian this system is equivalent to the some system $S'(x)$. Rewrite it by Theorem \ref{theorem_system_nf}. This one contains only one variable, then equations in $S'(x)$ can be one of:
\begin{enumerate}
\item $x \leq a$,
\item $x \wedge a = 0$,
\item $a \leq x$,
\end{enumerate}
where $a > 0$. Denote that $x = 0$ is a solution of $S(x)$, clearly $x = 0$ is a solution $S'(x)$ as well. Hence, there is no equations $a \leq x$ in $S'(x)$. According to this $S'(x)$ has following form:
\begin{equation*}
\left\{
\begin{array}{l}
x \leq b_1, \\
\ldots \\
x \leq b_m, \\
x \wedge a_1 = 0, \\
\ldots \\
x \wedge a_n = 0.
\end{array}
\right.
\end{equation*}

Let $b = b_1 \wedge \ldots b_m$ and $a = a_1 \vee \ldots \vee a_n$, then $S'(x)$ is equivalent to:
\begin{equation*}
\left\{
\begin{array}{l}
x \leq b, \\
x \wedge a = 0.
\end{array}
\right.
\end{equation*}

If this system consists of unique equation $x \leq b$, then required equivalence has been proved.

If this system consists of unique equation $x \wedge a = 0$, we have conflict with $\{ c_j | j \in J\}$ is unbound~--- because it it easy to show that for all $j \in J$ is $c_j \leq a$.

If this system consists of both equations:
\begin{equation*}
\left\{
\begin{array}{l}
x \leq b, \\
x \wedge a = 0
\end{array}
\right.
\sim x \leq b \setminus a.
\end{equation*}

Let $c = b \setminus a$. Weak equationally Noetherian property implies for any upper-unbounded set $\{ c_j | j \in J\} \subseteq \mathcal{C}$:
$$
\left\{ x \wedge c_j = 0 | j \in J \right\} \, \sim \, x \leq c.
$$
\end{proof}

\begin{proposition}
Let $\A$ is an Ershov $\mathcal{C}$-algebra. If $\A$ is weak equationally noetherian, the for any bounded in $\A$ subset of constants from $\mathcal{C}$ supremum in $\A$ exists and also is in $\mathcal{C}$.
\end{proposition}
\begin{proof}
Let $\{ c_j | j \in J\} \subseteq \mathcal{C}$ an arbitrary set f constants which is upper-bounded in $\A$ with element $e$. Consider following system of equations:
$$
S(x) = \{ c_j \leq x | j \in J\}.
$$

Denote that $e$ is a solution of $S(x)$. By the weak equationally Noetherian property for any system of equations $S(x)$ exists some system $S'(x)$, for which:
\begin{equation*}
S'(x) = \left\{
\begin{array}{l}
x \leq a_1, \\
\ldots \\
x \leq a_n, \\
x \wedge b_1 = 0, \\
\ldots \\
x \wedge b_m = 0, \\
d_1 \leq x, \\
\ldots \\
d_k \leq x.
\end{array}
\right.
\end{equation*}

Let $a = a_1 \wedge \ldots \wedge a_n$, $b = b_1 \vee \ldots \vee b_m$, $d = d_1 \vee \ldots \vee d_k$. Rewrite $S'(x)$:
\begin{equation*}
S'(x) = \left\{
\begin{array}{l}
x \leq a, \\
x \wedge b = 0, \\
d \leq x.
\end{array}
\right.
\end{equation*}

Equation $x \leq a$ is not really in $S'(x)$. If $a$ is a greatest element in $\mathcal{C}$ and $\A$, then this equation are logically true and it can be missed. Otherwise, there is an element $a' \in \A \setminus \mathcal{C}$, that $a < a'$, and systems $S(x)$ and $S'(x)$ are not equivalent, because $e \vee a'$ is a solution of $S(x)$ but not $S'(x)$.

In addition, equation $x \wedge b = 0$ is also not in $S'(x)$. Otherwise, elemenet $e \vee b$ is a solution of $S(x)$ but not $S'(x)$.

Finally, $S'(x)$ consists only of one equation $d \leq x$, and it is simple to show that $d = \sup \{c_j | j \in J\}$.
\end{proof}

\begin{remark}
If there is a supremum in $\A$ for any set of constants from $\mathcal{C}$ and supremum is in $\mathcal{C}$, then there is an infinum for this set in $\A$, and infinum is also in $\mathcal{C}$.
\end{remark}
\begin{proof}
Let $\{ c_j | j \in J\} \subseteq \mathcal{C}$ an arbitrary set of constants. Here we will show that this set has an infinum and supremum in $\A$ and these ones are also in $\mathcal{C}$. By the condition there is the supremum $\sup \{ c_j | j \in J\} = d$ in $\A$ and $d \in \mathcal{C}$. Let's consider the set $\{ d \setminus c_j | j \in J\} \subseteq \mathcal{C}$. By the condition, this set laso have supremum $\sup \{ d \setminus c_j | j \in J\} = d' \in \mathcal{C}$. It implies that $d \setminus d'$ is infinum for $\{ c_j | j \in J\}$, and it is clearly in $\mathcal{C}$.
\end{proof}

Finally prove a criterion of weak equationally Noetherian property.

\begin{theorem}
Let $\A$ is an Ershov $\mathcal{C}$-algebra. $\A$ is weak equationally Noetherian by equations, if and only if:
\begin{enumerate}
\item for each set of constants from $\mathcal{C}$ bounded in $\A$ exists supremum in $\A$ and supremum is also in $\mathcal{C}$;
\item for each upper-unbounded set $\{ c_j | j \in J\} \subseteq \mathcal{C}$ and some $c \in \mathcal{C}$:
$$
\left\{ x \wedge c_j = 0 | j \in J \right\} \, \sim \, x \leq c.
$$
\end{enumerate}
\end{theorem}
\begin{proof}
One part of this theorem follows from propositions above. Let's prove another part.
Consider an arbitrary system $S(\ox)$ of variables $\ox$. Rewrite this system by Theorem \ref{theorem_system_nf} and make equivalent system $S'(x)$ consisted from equiations in following groups:
\begin{enumerate}
\item $x_{i_1} \wedge \ldots \wedge x_{i_m} \wedge c \leq x_{j_1} \vee \ldots \vee x_{j_l},$
\item $x_{i_1} \wedge \ldots \wedge x_{i_m} \wedge c = 0,$
\item $x_{i_1} \wedge \ldots \wedge x_{i_m} \leq x_{j_1} \vee \ldots \vee x_{j_l} \vee c, $
\item $x_{i_1} \wedge \ldots \wedge x_{i_m} \leq x_{j_1} \vee \ldots \vee x_{j_l}.$
\end{enumerate}
Fix two nonintersecting set of variables $X = \{ x_{i_1}, \ldots, x_{i_m} \}$ and $Y = \{x_{j_1}, \ldots, x_{j_l}\}$ and consider equations from the first group in $S'(x)$, which have variables $X$ on the left side and $Y$ on the right:
$$
\left\{ \bigwedge X \wedge c_j \leq \bigvee Y \, | \, j \in J \right\}.
$$
If set of constants $\{ c_j | j \in J\}$ is upper-bounded in $\A$, then by previous proposition for there is upper-bound constant $c$, and we can rewrite whole group of equations as one:
$$
\left\{ \bigwedge X \wedge c_j \leq \bigvee Y \, | \, j \in J \right\}
\sim \bigwedge X \wedge c \leq \bigvee Y.
$$
Otherwise, rewrite each equation:
\begin{equation*}
\begin{array}{c}
\bigwedge X \wedge c_j \leq \bigvee Y \sim \\
\sim \left( \bigwedge X \wedge c_j \right) \setminus \bigvee Y \leq 0 \sim \\
\sim \left( \bigwedge X \setminus \bigvee Y \right) \wedge c_j \leq 0 \sim \\
\sim \left( \bigwedge X \setminus \bigvee Y \right) \wedge c_j = 0.
\end{array}
\end{equation*}

By the condition of this theorem for some $c \in \mathcal{C}$:
$$
\left\{
\left( \bigwedge X \setminus \bigvee Y \right) \wedge c_j = 0
\right\} \sim
\left( \bigwedge X \setminus \bigvee Y \right) \leq c.
$$

Therefore, we replace one group of equations by one equation. Analogically we can replace group $x_{i_1} \wedge \ldots \wedge x_{i_m} \wedge c = 0$ by one equation.

Consider second group for each two nonintersecting set of variables $X$ and $Y$:
$$
\left\{ \bigwedge X \leq \bigvee Y \vee c_j \, | \, j \in J \right\}.
$$
By remark above $\{ c_j | j \in J\}$ have infinum $c$ in $\A$ and $c \in \mathcal{C}$. Clearly:
$$
\left\{ \bigwedge X \leq \bigvee Y \vee c_j \, | \, j \in J \right\} \sim
 \bigwedge X \leq \bigvee Y \vee c.
$$

Denote that for arbitrary  $X$ and $Y$ there is can be only
$$
x_{i_1} \wedge \ldots \wedge x_{i_m} \leq x_{j_1} \vee \ldots \vee x_{j_l}.
$$

Replace all groups by required equation and we take finite equivalent system of equations. Finally, Ershov $\mathcal{C}$-algebra $\A$ is weak equationally Noetherian.
\end{proof}

\end{document}